\documentclass[11pt]{article}

\usepackage[margin=1in]{geometry}
\usepackage{amsmath,amssymb,amsthm}
\usepackage{booktabs}
\usepackage{array}
\usepackage{tikz}
\usetikzlibrary{arrows.meta}
\usepackage{hyperref}
\hypersetup{hidelinks}

\definecolor{ABred}{HTML}{A23B45}
\definecolor{ABblue}{HTML}{2F6690}
\definecolor{ABgreen}{HTML}{3D7A57}
\definecolor{ABgold}{HTML}{E9C46A}
\definecolor{ABink}{HTML}{263238}
\tikzset{
  ab vertex/.style={circle,draw=ABink,fill=white,line width=.45pt,
    minimum size=4.6mm,inner sep=0pt,font=\scriptsize},
  ab terminal/.style={ab vertex,draw=ABred,fill=ABred,text=white,
    line width=.9pt},
  ab selected/.style={ab vertex,draw=ABgreen,fill=ABgold!70,
    line width=.9pt},
  ab faint edge/.style={draw=ABink!52,line width=.45pt},
  ab strong edge/.style={draw=ABred,line width=1.35pt},
  ab annulus edge/.style={draw=ABblue,line width=1.05pt}
}

\newtheorem{theorem}{Theorem}[section]
\newtheorem{lemma}[theorem]{Lemma}
\newtheorem{proposition}[theorem]{Proposition}
\newtheorem{corollary}[theorem]{Corollary}
\theoremstyle{definition}
\newtheorem{construction}[theorem]{Construction}
\theoremstyle{remark}
\newtheorem{remark}[theorem]{Remark}

\newcommand{\ind}{\mathbf 1}

\title{A $15/31$ Counterexample Family\\
to the Albertson--Berman Conjecture}
\author{Heejae Jung}
\date{August 11, 2026}

\begin{document}
\maketitle

\begin{abstract}
For a graph $G$, let $a(G)$ be the maximum number of vertices in an induced forest.
The Albertson--Berman conjecture, posed in 1979 \cite{AlbertsonBerman1979}, asserts
that every $n$-vertex planar graph satisfies $a(G)\ge n/2$. 
Borodin's bound $a(G)\ge 2n/5$ remains the general lower bound 
toward this problem\cite{Borodin1979}.
We disprove the conjecture with an explicit 31-vertex plane 
triangulation $T$ satisfying $a(T)=15$. Moreover, for every integer
$k\ge2$, we construct a simple planar graph $M_k$ with
\[
 |V(M_k)|=31k,\qquad a(M_k)=15k,
\]
so that $a(M_k)/|V(M_k)|=15/31<1/2$.  Every member of the family has minimum
degree five.  The construction starts from a $31$-vertex seed obtained by
substituting a $14$-vertex two-terminal gadget into a pentagonal bipyramid, and
then uses annular joins along facial triangles to preserve the exact ratio.
The resulting graphs are sphere triangulations, and hence maximal planar.
\end{abstract}

\section{Introduction}
For a graph $G$, let $a(G)$ denote the maximum number of vertices in an induced
forest of $G$.  In 1979, Albertson and Berman conjectured that every planar graph
$G$ satisfies\cite{AlbertsonBerman1979}
\begin{equation}\label{eq:AB}
a(G) \ge \frac{|V(G)|}{2}.
\end{equation}
Borodin's acyclic $5$-coloring theorem gives the general lower bound
$a(G)\ge2|V(G)|/5$\cite{Borodin1979}.  For several subclasses, stronger bounds
are known: Hosono proved a $2n/3$ bound for outerplanar graphs\cite{Hosono1990};
Salavatipour proved that every triangle-free planar graph has an induced forest
of order at least $(17n+24)/32$\cite{Salavatipour2006}, and Kowalik, Lu\v{z}ar,
and \v{S}krekovski later improved the triangle-free bound to
$(71n+72)/128$\cite{KowalikLuzarSkrekovski2010}.  Akiyama and Watanabe's
$5n/8$ conjecture for bipartite planar graphs is a related stronger benchmark
for that subclass\cite{AkiyamaWatanabe1987}.

A 2025 account still recorded the simple planar conjecture as
open\cite{EnamiMatsumotoYashima2025}.  In 2026, Makarov studied planar
multigraphs and, in the variant with no $2$-faces, constructed an infinite
sequence with maximum induced-forest order
\[
 a(M)=\frac{3n}{7}+\frac{4}{7},
\]
whose asymptotic induced-forest ratio is $3/7$\cite{Makarov2026}.  That result
concerns multigraphs and does not provide a counterexample in the original
simple-graph setting.

Here we disprove the Albertson--Berman conjecture with an explicit infinite
family having exact ratio $15/31$.  The family is built from a two-terminal
gadget with an even number of internal vertices and an internal maximum induced
forest of exactly half its internal order.  The main result is the following.

\begin{theorem}[Main theorem]\label{thm:main}
For every integer $k\ge2$ there exists a simple planar graph $M_k$ such that
\[
 |V(M_k)|=31k,
 \qquad
 a(M_k)=15k,
 \qquad
 \frac{a(M_k)}{|V(M_k)|}=\frac{15}{31},
\]
and
\[
 \delta(M_k)=5.
\]
\end{theorem}

Since $15/31<1/2$, Theorem~\ref{thm:main} disproves \eqref{eq:AB}.
The graphs produced by the construction are sphere triangulations, hence
maximal planar (Lemma~\ref{lem:topology}).

The construction proceeds as follows.  We first design a $14$-vertex
two-terminal plane triangulation~$X$ whose induced-forest capacity drops by
exactly one when both terminals are selected
(Lemma~\ref{lem:profile}).  This one-unit penalty is the local mechanism
behind the counterexample: substituting a copy of~$X$ onto a base edge
converts an edge constraint into a vertex deficit.  A transfer law
(Theorem~\ref{thm:transfer}) makes this accounting precise for any choice of
base graph and decorated edge set.  When the base graph is the pentagonal
bipyramid and only two disjoint rim edges are decorated, the resulting
$31$-vertex graph~$T$ achieves $a(T)=15$
(Proposition~\ref{prop:seed}).  Finally, copies of~$T$ are joined along
facial triangles that avoid the forest witness, so that the ratio $15/31$
is preserved exactly while the minimum degree rises to five.

\section{The two-terminal triangulation}

The key local ingredient is a plane triangulation with
two distinguished adjacent vertices (terminals) whose induced-forest capacity
drops by exactly one when both terminals are forced into the forest. We now
construct such a gadget from the icosahedral graph.

Let $X$ have vertex set
\[
 V(X)=\{a,b,c,d,e,f,g,h,i,j,k,l,m,n\}.
\]
The following table gives cyclic neighbour lists in a spherical embedding.
The distinguished terminals are the adjacent vertices $g$ and $h$.
Figure~\ref{fig:gadget} gives a Schlegel drawing of the same rotation system.

\begin{center}
\begin{tabular}{c@{\quad}l c@{\qquad}c@{\quad}l}
\toprule
$a$ & $b,c,d,e,f,g$ && $h$ & $b,g,f,i$\\
$b$ & $a,g,h,i,j,c$ && $i$ & $b,h,f,m,n,j$\\
$c$ & $a,b,j,k,d$ && $j$ & $b,i,n,k,c$\\
$d$ & $a,c,k,l,e$ && $k$ & $c,j,n,l,d$\\
$e$ & $a,d,l,m,f$ && $l$ & $d,k,n,m,e$\\
$f$ & $a,e,m,i,h,g$ && $m$ & $e,l,n,i,f$\\
$g$ & $a,f,h,b$ && $n$ & $i,m,l,k,j$\\
\bottomrule
\end{tabular}
\end{center}

The embedding has $36$ edges and the following $24$ triangular faces:
\[
\begin{gathered}
abc,abg,acd,ade,aef,afg,bcj,bgh,bhi,bij,cdk,cjk,\\
del,dkl,efm,elm,fgh,fhi,fim,ijn,imn,jkn,kln,lmn.
\end{gathered}
\]
The graph is connected, each displayed cyclic neighbour list is a single
vertex link, every edge occurs in two of these faces, and
$14-36+24=2$.  Thus the displayed rotation system is a triangulation of the
sphere.  In particular, $X$ is simple and planar, and the terminal edge
$gh$ has facial triangles $ghb$ and $ghf$.

\begin{figure}[htbp]
\centering
\begin{tikzpicture}[x=.78cm,y=.78cm,line cap=round,line join=round]
  \coordinate (a) at (-5.000,-3.000);
  \coordinate (b) at (-1.756, 0.089);
  \coordinate (c) at (-1.546,-1.105);
  \coordinate (d) at (-0.095,-1.771);
  \coordinate (e) at ( 5.000,-3.000);
  \coordinate (f) at ( 0.000, 5.660);
  \coordinate (g) at (-1.941, 1.210);
  \coordinate (h) at (-1.007, 2.092);
  \coordinate (i) at (-0.331, 1.407);
  \coordinate (j) at (-0.710,-0.068);
  \coordinate (k) at (-0.171,-0.775);
  \coordinate (l) at ( 1.244,-0.975);
  \coordinate (m) at ( 1.233, 0.627);
  \coordinate (n) at ( 0.253, 0.043);

  \fill[ABred!9] (g)--(h)--(b)--cycle;
  \fill[ABred!9] (g)--(h)--(f)--cycle;
  \foreach \u/\v in {
    a/b,a/c,a/d,a/e,a/f,a/g,
    b/g,b/h,b/i,b/j,b/c,
    c/j,c/k,c/d,d/k,d/l,d/e,
    e/l,e/m,e/f,f/m,f/i,f/h,f/g,
    g/h,h/i,i/m,i/n,i/j,j/n,j/k,
    k/n,k/l,l/n,l/m,m/n}
    \draw[ab faint edge] (\u)--(\v);
  \draw[ab strong edge] (g)--(h);

  \foreach \v in {a,b,c,d,e,f,i,j,k,l,m,n}
    \node[ab vertex] at (\v) {$\v$};
  \foreach \v in {g,h}
    \node[ab terminal] at (\v) {$\v$};
  \node[font=\scriptsize,text=ABred] at (-2.30,2.02) {$ghb$};
  \node[font=\scriptsize,text=ABred] at (-.18,2.96) {$ghf$};
  \node[font=\scriptsize,text=ABink!70] at (0,-3.43)
    {outer face $aef$};
\end{tikzpicture}
\caption{A plane drawing of the two-terminal gadget $X$.  The terminals
$g,h$, their common edge, and the two incident faces are highlighted. 
The cyclic neighbour lists above determine the embedding. }
\label{fig:gadget}
\end{figure}
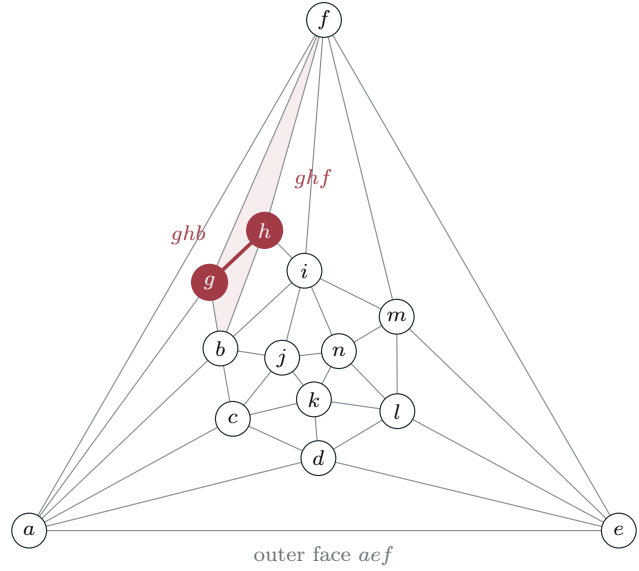

Put $J=X-\{g,h\}$.  For $A\subseteq\{g,h\}$ define the internal forest
capacity
\[
 p(A)=\max\bigl\{|Y|:Y\subseteq V(J),\ X[A\cup Y]
                         \text{ is a forest}\bigr\}.
\]

\subsection{Properties of the icosahedral graph}

The graph $I=J+bf$ is the icosahedral graph.  We use three elementary
properties, all immediately checkable from the displayed neighbour lists.

\begin{lemma}[Icosahedral edge bounds]\label{lem:ico}
In the icosahedral graph $I$:
\begin{enumerate}
\item every five vertices span at least three edges;
\item every six vertices span at least five edges;
\item every six vertices containing the face $abf$ or the face $bfi$ span
      at least six edges.
\end{enumerate}
\end{lemma}

\begin{proof}
Fix a vertex $x$ of $I$.  Its six nonneighbours induce a wheel: one hub
joined to a $5$-cycle.  Each neighbour of $x$ has two neighbours on that
$5$-cycle.  These facts are the usual two-pentagon description of the
icosahedron and can also be read directly from the table above.  In
particular, the wheel has independence number two, so $\alpha(I)\le3$;
the independent set $\{a,i,k\}$ gives equality.

Let $R$ be a five-set.  If $I[R]$ had at most two edges, then, because
$\alpha(I)=3$, its two edges would have to be disjoint and the fifth vertex
$x$ would be isolated.  The other four vertices would lie in the
nonneighbour wheel of $x$ and would span only two edges.  This is impossible:
four rim vertices span three edges, while a set consisting of the hub and
three rim vertices spans at least three.  This proves (1).

Now let $R$ have six vertices.  If $e_I(R)\le4$, then $I[R]$ has a vertex
$x$ of degree at most one.  If its degree is zero, the other five vertices
lie in the nonneighbour wheel of $x$ and span at least five edges.  If its
degree is one, with neighbour $y$, the other four vertices span at least
three edges in that wheel.  Equality there forces four rim vertices.  But
$y$, being a neighbour of $x$, is adjacent to two rim vertices and hence to
at least one of the selected four.  Together with $xy$, this gives at least
five edges, a contradiction.  Thus (2) holds.

For (3), first use the face $F=\{a,b,f\}$.  The other nine vertices split
into
\[
 A=\{c,e,i\},\qquad B=\{d,j,m\},\qquad C=\{k,l,n\}.
\]
Each vertex of $A$ has two neighbours in $F$, each vertex of $B$ has one,
and vertices of $C$ have none.  The set $C$ induces a triangle, and the
$B$--$C$ edges are
\[
 dk,dl,jk,jn,ml,mn.
\]
Choose any three vertices outside $F$.  If $t$ of them lie in $C$, then:
for $t=0$ their attachments to $F$ contribute at least three edges; for
 $t=1$, the only tight attachment case chooses two vertices of $B$, and the
 chosen $C$-vertex meets at least one of them.  For $t=2$, the two chosen
 vertices of $C$ supply their mutual edge.  If the third vertex lies in $A$,
 its two attachments to $F$ complete the required three edges; if it lies in
 $B$, use its one attachment and at least one edge from it to the two chosen
 vertices of $C$.  For $t=3$, the triangle $C$ contributes three.  Adding the
 three edges of $F$ proves the claim for $abf$.  The automorphism
\[
 (a\ i)(b\ f)(c\ m)(d\ n)(e\ j)(k\ l)
\]
maps $abf$ to $bfi$, proving the other case.
\end{proof}

\subsection{A symbolic proof of the terminal profile}

\begin{lemma}[Terminal profile]\label{lem:profile}
The internal forest profile of $X$ is
\[
 \boxed{
 p(\varnothing)=p(\{g\})=p(\{h\})=6,
 \qquad
 p(\{g,h\})=5.}
\]
\end{lemma}

\begin{proof}
Let $Y\subseteq V(J)$ have seven vertices, and put
$R=V(I)\setminus Y$, so $|R|=5$.  The graph $I$ is $5$-regular with
$30$ edges.  Therefore
\[
 e_I(Y)=30-5|R|+e_I(R)=5+e_I(R)\ge8
\]
by Lemma~\ref{lem:ico}(1).  Since $J=I-bf$, we have $e_J(Y)\ge7$.
A seven-vertex forest has at most six edges, so $J[Y]$ contains a cycle.
The same cycle remains after selecting either one terminal.  Hence
\[
 p(\varnothing),p(\{g\}),p(\{h\})\le6.
\]

Now select both terminals and suppose that six internal vertices $Y$ are
also selected.  If $b\in Y$ or $f\in Y$, then the selected graph contains
the triangle $ghb$ or $ghf$.  We may therefore assume
$Y\subseteq V(I)\setminus\{b,f\}$.  By Lemma~\ref{lem:ico}(2),
$e_I(Y)\ge5$.  If $e_I(Y)\ge6$, then $I[Y]=J[Y]$ already contains a cycle.
Thus a possible forest would require $e_I(Y)=5$.

If $a\notin Y$, then the complementary six-set $V(I)\setminus Y$ contains
the face $abf$; if $i\notin Y$, it contains the face $bfi$.  In a
$5$-regular graph, complementary vertex sets of equal size span the same
number of edges, because their degree sums have the same cross-edge term.
Lemma~\ref{lem:ico}(3) would then give $e_I(Y)\ge6$.  Consequently any
remaining case has $a,i\in Y$.  If $I[Y]$ were acyclic with six vertices and
five edges, it would be a tree and would contain an $a$--$i$ path.  The
three edges $ga,gh,hi$ close that path to a cycle.  Therefore
$p(\{g,h\})\le5$.

The matching lower bounds are given by the following induced trees:
\begin{center}
\begin{tabular}{c@{\qquad}l@{\qquad}l}
\toprule
selected terminals & selected internal vertices & induced tree\\
\midrule
$\varnothing$ & $b,c,e,f,i,k$ & $e-f-i-b-c-k$\\
$g$ & $a,c,e,i,j,l$ & edges $ga,ac,ae,cj,ji,el$\\
$h$ & $a,c,e,i,j,l$ & $h-i-j-c-a-e-l$\\
$g,h$ & $c,d,e,i,j$ & $g-h-i-j-c-d-e$\\
\bottomrule
\end{tabular}
\end{center}
This proves all four equalities.
\end{proof}

\section{Selected-edge substitution}

Let $B$ be a finite simple graph, let $Q\subseteq E(B)$, and put $q=|Q|$.
For each edge $uv\in Q$, take a fresh copy of $X$, identify its terminals
$g,h$ with $u,v$, and identify its terminal edge with the existing edge
$uv$.  All internal vertices of different copies remain distinct, and all
edges of $B$ are retained.  Denote the resulting graph by $X(B,Q)$.
For $S\subseteq V(B)$ write
\[
 e_Q(S)=|\{uv\in Q:u,v\in S\}|.
\]
The basic counts are
\begin{equation}\label{eq:subcounts}
 |V(X(B,Q))|=|V(B)|+12q,
 \qquad
 |E(X(B,Q))|=|E(B)|+35q.
\end{equation}
If $B$ is planar, then $X(B,Q)$ is planar:
each step is a planar $2$-clique sum along an edge, with the shared edge
retained rather than duplicated.  Figure~\ref{fig:substitution} isolates the
same local operation and the one-unit charge that drives the transfer law.

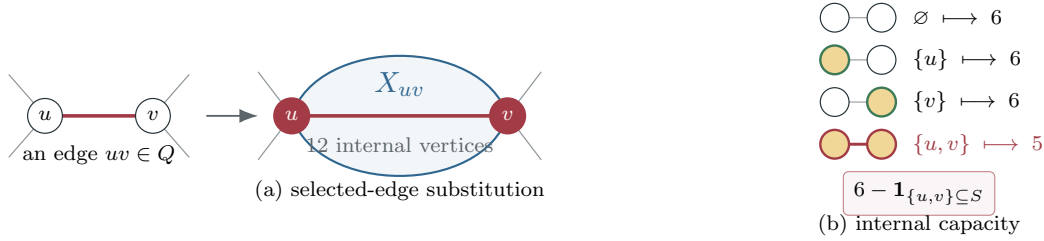
\begin{figure}[htbp]
\centering
\begin{minipage}[c]{.56\textwidth}
\centering
\begin{tikzpicture}[x=.68cm,y=.68cm,line cap=round,line join=round]
  \coordinate (lu) at (0,0);
  \coordinate (lv) at (2.1,0);
  \draw[ab faint edge] (-.7,.8)--(lu)--(-.7,-.8);
  \draw[ab faint edge] (2.8,.8)--(lv)--(2.8,-.8);
  \draw[ab strong edge] (lu)--(lv);
  \node[ab vertex] at (lu) {$u$};
  \node[ab vertex] at (lv) {$v$};
  \node[font=\scriptsize] at (1.05,-.75) {an edge $uv\in Q$};

  \draw[-{Latex[length=2.5mm]},line width=.7pt,ABink!75]
    (3.15,0)--(4.15,0);

  \coordinate (ru) at (4.8,0);
  \coordinate (rv) at (9.0,0);
  \path[draw=ABblue,fill=ABblue!6,line width=.8pt]
    (ru) to[out=72,in=108] (rv) to[out=-108,in=-72] (ru)--cycle;
  \draw[ab strong edge] (ru)--(rv);
  \draw[ab faint edge] (4.15,.85)--(ru)--(4.15,-.85);
  \draw[ab faint edge] (9.65,.85)--(rv)--(9.65,-.85);
  \node[ab terminal] at (ru) {$u$};
  \node[ab terminal] at (rv) {$v$};
  \node[font=\normalsize,text=ABblue] at (6.9,.62) {$X_{uv}$};
  \node[font=\scriptsize,text=ABink!75] at (6.9,-.56)
    {$12$ internal vertices};
  \node[font=\scriptsize] at (6.9,-1.45) {(a) selected-edge substitution};
\end{tikzpicture}
\end{minipage}\hfill
\begin{minipage}[c]{.40\textwidth}
\centering
\begin{tikzpicture}[x=.62cm,y=.62cm,line cap=round]
  \foreach \y in {3.0,2.1,1.2,.3}
    \draw[ABink!45,line width=.45pt] (0,\y)--(1,\y);

  \node[ab vertex,minimum size=3.8mm] at (0,3.0) {};
  \node[ab vertex,minimum size=3.8mm] at (1,3.0) {};
  \node[font=\scriptsize,anchor=west] at (1.45,3.0)
    {$\varnothing\;\longmapsto\;6$};

  \node[ab selected,minimum size=3.8mm] at (0,2.1) {};
  \node[ab vertex,minimum size=3.8mm] at (1,2.1) {};
  \node[font=\scriptsize,anchor=west] at (1.45,2.1)
    {$\{u\}\;\longmapsto\;6$};

  \node[ab vertex,minimum size=3.8mm] at (0,1.2) {};
  \node[ab selected,minimum size=3.8mm] at (1,1.2) {};
  \node[font=\scriptsize,anchor=west] at (1.45,1.2)
    {$\{v\}\;\longmapsto\;6$};

  \draw[ABred,line width=1.1pt] (0,.3)--(1,.3);
  \node[ab selected,draw=ABred,minimum size=3.8mm] at (0,.3) {};
  \node[ab selected,draw=ABred,minimum size=3.8mm] at (1,.3) {};
  \node[font=\scriptsize,anchor=west,text=ABred] at (1.45,.3)
    {$\{u,v\}\;\longmapsto\;5$};

  \node[draw=ABred!65,fill=ABred!5,rounded corners=2pt,
    font=\scriptsize,inner sep=4pt] at (1.8,-.75)
    {$6-\ind_{\{u,v\}\subseteq S}$};
  \node[font=\scriptsize] at (1.8,-1.42) {(b) internal capacity};
\end{tikzpicture}
\end{minipage}
\caption{The local accounting behind Theorem~\ref{thm:transfer}.  A gadget
copy shares exactly the decorated base edge and its endpoints with the graph
already present.  Its internal forest capacity drops from six to five only
when both endpoints belong to the selected base set $S$.}
\label{fig:substitution}
\end{figure}

Define
\[
 \beta(B,Q)=
 \max\bigl\{|S|-e_Q(S):S\subseteq V(B),\ B[S]\text{ is a forest}\bigr\}.
\]

\begin{theorem}[Selected-edge transfer law]\label{thm:transfer}
For every finite simple graph $B$ and every $Q\subseteq E(B)$,
\[
 \boxed{a(X(B,Q))=6|Q|+\beta(B,Q).}
\]
\end{theorem}

\begin{proof}
Let $F\subseteq V(X(B,Q))$ induce a forest, and put
$S=F\cap V(B)$.  Since $B[S]$ is an induced subgraph of $X(B,Q)[F]$, it is
a forest.  In the copy on $uv\in Q$, Lemma~\ref{lem:profile} bounds the
number of selected internal vertices by
\[
 6-\ind_{\{u,v\}\subseteq S}.
\]
Summing over all copies gives
\[
 |F|\le |S|+6q-e_Q(S)\le6q+\beta(B,Q).
\]

Conversely, choose $S$ attaining $\beta(B,Q)$ and, in every gadget, choose
an equality witness from Lemma~\ref{lem:profile} for the prescribed terminal
state.  Begin with the base forest $B[S]$ and add the local forests one at a
time.  The intersection of a local forest with the graph already present is
empty, a single terminal, or the connected edge $uv$, and no internal
vertex of the gadget has a neighbour outside its copy.  The union of two
forests along a connected subtree is a forest.  Induction over the gadgets
therefore gives an induced forest with
\[
 |S|+6q-e_Q(S)=6q+\beta(B,Q)
\]
vertices.
\end{proof}

\subsection{The full-edge specialization}

If $Q=E(B)$ and $B[S]$ is a forest, then
\[
 |S|-e_B(S)=c(B[S])\le\alpha(B),
\]
because one vertex chosen from each component is independent; equality is
attained by an independent set.  Hence Theorem~\ref{thm:transfer} recovers
\begin{equation}\label{eq:full}
 a(X(B,E(B)))=6|E(B)|+\alpha(B).
\end{equation}
For comparison, the choice $B=K_4$ gives a $76$-vertex block with induced-
forest number $37$, hence ratio $37/76>15/31$.  The improvement below comes
from retaining undecorated base edges as cycle constraints without paying the
$12$ internal vertices of a gadget on each such edge.

\section{The pentagonal-bipyramid obstruction}

The base graph must be small enough to keep the vertex count low, yet rich
enough in short cycles to constrain the decorated-edge penalty. The
pentagonal bipyramid achieves both.

Let $B=C_5*\overline{K_2}$ be the pentagonal bipyramid.  Its rim vertices
are $0,1,2,3,4$ in cyclic order; its two nonadjacent apices are $5,6$; and
each apex is adjacent to every rim vertex.  Thus $B$ is a plane
triangulation with seven vertices and fifteen edges.  Decorate only the two
rim edges
\[
 Q=\{01,23\}.
\]
Figure~\ref{fig:bipyramid} shows the resulting seven-vertex transfer core.

\begin{figure}[htbp]
\centering
\begin{tikzpicture}[x=1.05cm,y=1.05cm,line cap=round,line join=round]
  \coordinate (v0) at (-3.2,-2.2);
  \coordinate (v1) at ( 3.2,-2.2);
  \coordinate (v2) at ( .8, .346);
  \coordinate (v3) at ( 0,  .856);
  \coordinate (v4) at (-.8, .346);
  \coordinate (v5) at ( 0, -.570);
  \coordinate (v6) at ( 0, 3.3);

  \foreach \u/\v in {0/1,1/2,2/3,3/4,4/0,
    5/0,5/1,5/2,5/3,5/4,
    6/0,6/1,6/2,6/3,6/4}
    \draw[ab faint edge] (v\u)--(v\v);

  \draw[ABgreen,line width=1.2pt] (v0)--(v5)--(v2);
  \draw[ab strong edge] (v0)--(v1);
  \draw[ab strong edge] (v2)--(v3);

  \foreach \v in {1,3,4}
    \node[ab vertex] at (v\v) {$\v$};
  \node[ab vertex,fill=ABblue!12] at (v6) {$6$};
  \foreach \v in {0,2,5}
    \node[ab selected] at (v\v) {$\v$};

  \node[font=\scriptsize,text=ABblue,anchor=west] at (.22,3.3) {apex};
  \node[font=\scriptsize,text=ABblue,anchor=west] at (.22,-.57) {apex};
  \node[draw=ABred!65,fill=ABred!5,rounded corners=2pt,
    font=\scriptsize,inner sep=4pt] at (-2.05,2.35)
    {$Q=\{01,23\}$};
  \node[draw=ABgreen!70,fill=ABgreen!6,rounded corners=2pt,
    font=\scriptsize,inner sep=4pt] at (2.05,2.35)
    {$S=\{0,2,5\}$,\quad $\phi(S)=3$};
  \node[font=\scriptsize,text=ABink!70] at (0,-2.68)
    {outer face $016$};
\end{tikzpicture}
\caption{A Schlegel drawing of the pentagonal bipyramid
$B=C_5*\overline{K_2}$.  The decorated matching $Q$ is red.  The gold
vertices and green edges show the equality witness $S=\{0,2,5\}$, whose
induced subgraph is a two-edge path.}
\label{fig:bipyramid}
\end{figure}
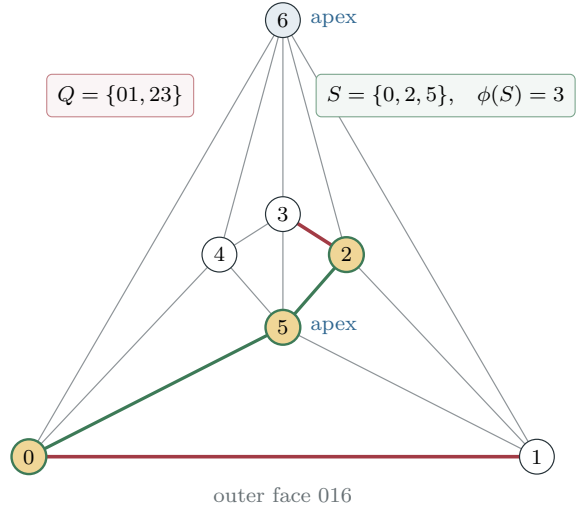

For $S\subseteq V(B)$ put
\[
 \phi(S)=|S|
 -\ind_{\{0,1\}\subseteq S}
 -\ind_{\{2,3\}\subseteq S}.
\]

\begin{lemma}[Bipyramid obstruction]\label{lem:bipyramid}
If $B[S]$ is a forest, then $\phi(S)\le3$.  Equality is attained.
Consequently $\beta(B,Q)=3$.
\end{lemma}

\begin{proof}
Suppose first that both apices are selected.  Two adjacent selected rim
vertices form a triangle with either apex, while two nonadjacent selected rim
vertices form a $4$-cycle with the two apices.  Hence at most one rim vertex
is selected, and $\phi(S)\le3$.

If exactly one apex is selected, then the selected rim vertices must be
independent in $C_5$, for an adjacent pair would form a triangle with the
apex.  There are at most two such rim vertices, again giving
$\phi(S)\le3$.

If no apex is selected, the assertion is immediate when at most three rim
vertices are selected.  A four-subset of the rim contains at least one of
the two matching edges $01,23$, and therefore pays at least one unit in
$\phi$.  All five rim vertices are forbidden because they induce $C_5$.
Thus $\phi(S)\le3$ in every case.  The set $S=\{0,2,5\}$ induces a path of
length two and has $\phi(S)=3$.
\end{proof}

Let
\[
 H=X(B,Q).
\]
Theorem~\ref{thm:transfer}, Lemma~\ref{lem:bipyramid}, and
\eqref{eq:subcounts} give
\begin{equation}\label{eq:H}
 |V(H)|=7+2\cdot12=31,
 \qquad
 |E(H)|=15+2\cdot35=85,
 \qquad
 a(H)=12+3=15.
\end{equation}

\begin{remark}[Odd bipyramids]\label{rem:odd}
For every $r \ge 2$, the same three-case argument applied to
$B_r=C_{2r+1}*\overline{K_2}$, with a maximum matching of the rim decorated,
gives
\[
 \beta(B_r,Q)=r+1,\qquad
 n_r=14r+3,\qquad
 a_r=7r+1.
\]
Thus $a_r/n_r=\frac12-1/[2(14r+3)]$, which is minimized at $r=2$.
\end{remark}

\section{The explicit 31-vertex seed}

To complete $H$ into a maximal planar graph, it remains to triangulate
the two quadrilateral faces left by the edge sums.  On each decorated
edge $(u,v)=(0,1),(2,3)$, map the local terminals $g,h$ to $u,v$ in that
order.  Map the local internal vertices
\[
 (a,b,c,d,e,f,i,j,k,l,m,n)
\]
to $7,8,\ldots,18$ in the first copy and to $19,20,\ldots,30$ in the
second.

The terminal edge $gh$ of $X$ has incident faces $ghb$ and $ghf$.
Perform the first edge sum so that the retained faces at the shared edge
$01$ are $(0,1,5)$ and $(0,1,8)$; the discarded faces $(0,1,6)$ and
$(0,1,12)$ merge
into the quadrilateral
\[
 (0,6,1,12).
\]
Perform the second edge sum so that the retained faces at $23$ are
$(2,3,5)$ and $(2,3,24)$; the other two faces merge into
\[
 (2,6,3,20).
\]
These are the only nontriangular faces.  As shown in
Figure~\ref{fig:completion}, add the two diagonals
\begin{equation}\label{eq:completion}
 6\,12,
 \qquad
 6\,20,
\end{equation}
and call the resulting graph $T$.

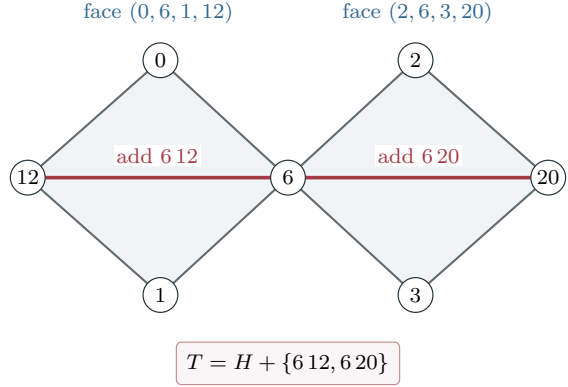
\begin{figure}[htbp]
\centering
\begin{tikzpicture}[x=1.35cm,y=1.35cm,line cap=round,line join=round]
  \coordinate (x6)  at ( 0, 0);
  \coordinate (x0)  at (-1.25, 1.15);
  \coordinate (x1)  at (-1.25,-1.15);
  \coordinate (x12) at (-2.55, 0);
  \coordinate (x2)  at ( 1.25, 1.15);
  \coordinate (x3)  at ( 1.25,-1.15);
  \coordinate (x20) at ( 2.55, 0);

  \fill[ABblue!7] (x0)--(x6)--(x1)--(x12)--cycle;
  \fill[ABblue!7] (x2)--(x6)--(x3)--(x20)--cycle;
  \draw[ABink!72,line width=.8pt] (x0)--(x6)--(x1)--(x12)--cycle;
  \draw[ABink!72,line width=.8pt] (x2)--(x6)--(x3)--(x20)--cycle;
  \draw[ab strong edge] (x6)--(x12);
  \draw[ab strong edge] (x6)--(x20);

  \foreach \name/\text in {x0/0,x1/1,x2/2,x3/3,x6/6,x12/12,x20/20}
    \node[ab vertex] at (\name) {$\text$};

  \node[font=\scriptsize,text=ABblue] at (-1.27,1.62)
    {face $(0,6,1,12)$};
  \node[font=\scriptsize,text=ABblue] at (1.27,1.62)
    {face $(2,6,3,20)$};
  \node[font=\scriptsize,text=ABred,fill=white,inner sep=1pt]
    at (-1.28,.22) {add $6\,12$};
  \node[font=\scriptsize,text=ABred,fill=white,inner sep=1pt]
    at (1.28,.22) {add $6\,20$};
  \node[draw=ABred!65,fill=ABred!5,rounded corners=2pt,
    font=\scriptsize,inner sep=4pt] at (0,-1.82)
    {$T=H+\{6\,12,6\,20\}$};
\end{tikzpicture}
\caption{The two nontriangular faces left by the edge sums; they share the
base apex $6$.  The red diagonals split both quadrilaterals into triangles,
giving the sphere-triangulated seed $T$.  The retained terminal edges $01$ and
$23$ are not boundary edges of these two faces.}
\label{fig:completion}
\end{figure}

\begin{proposition}[The seed]\label{prop:seed}
The graph $T$ is a simple maximal planar graph with
\[
 |V(T)|=31,
 \qquad
 |E(T)|=87,
 \qquad
 a(T)=15.
\]
Its degree multiset is
\[
 4^1\,5^{17}\,6^6\,7^7.
\]
\end{proposition}

\begin{proof}
The two edge sums give
\[
 |V(H)|=7+2(14-2)=31,
 \qquad
 |E(H)|=15+2(36-1)=85.
\]
 The connected graph $H$ contains neither diagonal in \eqref{eq:completion}.
 Adding them inside their two quadrilateral faces is visibly planar and
 preserves simplicity, so $T$ is simple and planar with
\[
 |E(T)|=87=3\cdot31-6.
\]
Hence every face is triangular and $T$ is maximal planar.

Since $H$ is a spanning subgraph of $T$, adding edges cannot increase the
induced-forest number, and \eqref{eq:H} gives $a(T)\le15$.  Conversely, the
vertices
\[
 W=\{1,3,5,7,9,11,13,14,16,20,22,23,26,27,29\}
\]
induce exactly the path
\[
 16-11-7-9-14-13-1-5-3-20-26-27-22-23-29.
\]
Thus $a(T)\ge15$, proving equality.

For the degree count, the gadget has eight internal vertices of degree five
and four of degree six.  Each decorated rim endpoint $0,1,2,3$ has degree
$4+4-1=7$ after the edge sum; rim vertex $4$ retains degree four; and the
apices initially have degree five.  The two completion edges raise apex $6$
from degree five to seven, first-copy vertex $12=f$ from six to seven, and
second-copy vertex $20=b$ from six to seven.  The asserted multiset follows.
\end{proof}

The seed $T$ has a unique vertex of degree four, namely vertex~$4$.
The annular amplification below eliminates this low-degree vertex in
every member of the family.

\section{Annular amplification}

To produce counterexamples of arbitrary size while maintaining the ratio
$15/31$ and minimum degree five, we join copies of $T$ along facial
triangles that are disjoint from the forest witness.

In the embedding of $T$, the ordered triples
\begin{equation}\label{eq:ports}
 P=(0,4,6),
 \qquad
 R=(2,19,24)
\end{equation}
are vertex-disjoint facial triangles.  Indeed, $P$ is the untouched base
face $046$, and $R$ is the face $gaf$ in the second gadget copy.  Both are
disjoint from the path witness $W$.  Their degree triples in $T$ are
\begin{equation}\label{eq:portdegrees}
 \deg_T(P)=(7,4,7),
 \qquad
 \deg_T(R)=(7,6,6).
\end{equation}

For disjoint ordered triangles
$U=(u_0,u_1,u_2)$ and $V=(v_0,v_1,v_2)$, define the annulus edges
\begin{equation}\label{eq:annulusedges}
 u_0v_0,
 u_1v_0,
 u_1v_1,
 u_2v_1,
 u_2v_2,
 u_0v_2.
\end{equation}
They support the six triangular faces
\begin{equation}\label{eq:annulusfaces}
\begin{gathered}
 u_0u_1v_0,
 \quad u_1v_0v_1,
 \quad u_1u_2v_1,
 \quad u_2v_1v_2,\\
 u_2u_0v_2,
 \quad u_0v_2v_0.
\end{gathered}
\end{equation}
Thus, after deleting the interiors of the two boundary faces, these six
triangles form a triangulated cylinder.  Reflecting one sphere if necessary
matches the boundary orientations.  Figure~\ref{fig:annulus-family}(a) shows
the six faces in the planar annular region.

\begin{construction}[The family $M_k$]\label{con:family}
Fix $k\ge2$ and take disjoint labelled copies $T_1,\ldots,T_k$ of $T$.
For $1\le i\le k-2$, insert the annulus
\eqref{eq:annulusedges} between $P_i$ and $R_{i+1}$.  Insert the final
annulus between $P_{k-1}$ and $P_k$.  Call the resulting graph $M_k$.
\end{construction}

The complete port schedule is shown in Figure~\ref{fig:annulus-family}(b).

\begin{figure}[htbp]
\centering
\begin{minipage}[c]{.42\textwidth}
\centering
\begin{tikzpicture}[x=.82cm,y=.82cm,line cap=round,line join=round]
  \coordinate (u0) at (0,3.0);
  \coordinate (u1) at (-2.8,-1.8);
  \coordinate (u2) at (2.8,-1.8);
  \coordinate (v0) at (0,.95);
  \coordinate (v1) at (-.95,-.65);
  \coordinate (v2) at (.95,-.65);

  \fill[ABblue!7]  (u0)--(u1)--(v0)--cycle;
  \fill[ABblue!13] (u1)--(v0)--(v1)--cycle;
  \fill[ABblue!7]  (u1)--(u2)--(v1)--cycle;
  \fill[ABblue!13] (u2)--(v1)--(v2)--cycle;
  \fill[ABblue!7]  (u2)--(u0)--(v2)--cycle;
  \fill[ABblue!13] (u0)--(v2)--(v0)--cycle;
  \fill[white] (v0)--(v1)--(v2)--cycle;

  \draw[ABink!78,line width=.85pt] (u0)--(u1)--(u2)--cycle;
  \draw[ABink!78,line width=.85pt] (v0)--(v1)--(v2)--cycle;
  \foreach \u/\v in {u0/v0,u1/v0,u1/v1,u2/v1,u2/v2,u0/v2}
    \draw[ab annulus edge] (\u)--(\v);

  \foreach \name/\text in
    {u0/u_0,u1/u_1,u2/u_2,v0/v_0,v1/v_1,v2/v_2}
    \node[ab vertex,minimum size=4.1mm] at (\name) {$\text$};
  \node[font=\scriptsize,text=ABink!70] at (0,-2.25)
    {(a) six triangular annulus faces};
\end{tikzpicture}
\end{minipage}\hfill
\begin{minipage}[c]{.55\textwidth}
\centering
\begin{tikzpicture}[x=.67cm,y=.67cm,line cap=round,line join=round]
  \draw[rounded corners=3pt,fill=ABblue!4,draw=ABink!65]
    (0,-1)--(1.6,-1)--(1.6,1)--(0,1)--cycle;
  \draw[rounded corners=3pt,fill=ABblue!4,draw=ABink!65]
    (2.8,-1)--(4.4,-1)--(4.4,1)--(2.8,1)--cycle;
  \draw[rounded corners=3pt,fill=ABblue!4,draw=ABink!65]
    (6.1,-1)--(7.9,-1)--(7.9,1)--(6.1,1)--cycle;
  \draw[rounded corners=3pt,fill=ABblue!4,draw=ABink!65]
    (9.1,-1)--(10.7,-1)--(10.7,1)--(9.1,1)--cycle;

  \node at (.8,.35) {$T_1$};
  \node at (3.6,.35) {$T_2$};
  \node at (7.0,.35) {$T_{k-1}$};
  \node at (9.9,.35) {$T_k$};
  \node[font=\large] at (5.25,0) {$\cdots$};

  \draw[ABgreen,line width=.85pt]
    (.35,-.45)--(.65,-.18)--(.95,-.55)--(1.25,-.28);
  \draw[ABgreen,line width=.85pt]
    (3.15,-.45)--(3.45,-.18)--(3.75,-.55)--(4.05,-.28);
  \draw[ABgreen,line width=.85pt]
    (6.45,-.45)--(6.75,-.18)--(7.08,-.55)--(7.52,-.28);
  \draw[ABgreen,line width=.85pt]
    (9.45,-.45)--(9.75,-.18)--(10.05,-.55)--(10.35,-.28);
  \node[font=\tiny,text=ABgreen] at (.8,-.75) {$W_1$};
  \node[font=\tiny,text=ABgreen] at (3.6,-.75) {$W_2$};
  \node[font=\tiny,text=ABgreen] at (7.0,-.75) {$W_{k-1}$};
  \node[font=\tiny,text=ABgreen] at (9.9,-.75) {$W_k$};

  \fill[ABblue] (1.6,0) circle (2.1pt);
  \fill[ABblue] (2.8,0) circle (2.1pt);
  \fill[ABblue] (4.4,0) circle (2.1pt);
  \fill[ABblue] (6.1,0) circle (2.1pt);
  \fill[ABred] (7.9,0) circle (2.1pt);
  \fill[ABred] (9.1,0) circle (2.1pt);
  \draw[ABblue,line width=1.2pt] (1.6,0)--(2.8,0);
  \draw[ABblue,line width=1.2pt,densely dotted] (4.4,0)--(6.1,0);
  \draw[ABred,line width=1.2pt] (7.9,0)--(9.1,0);

  \node[font=\scriptsize,text=ABblue,fill=white,inner sep=1pt]
    at (2.2,.38) {$P_1\! -\! R_2$};
  \node[font=\scriptsize,text=ABblue,fill=white,inner sep=1pt]
    at (5.25,.38) {$P_i\! -\! R_{i+1}$};
  \node[font=\scriptsize,text=ABred,fill=white,inner sep=1pt]
    at (8.5,.38) {$P_{k-1}\! -\! P_k$};
  \node[font=\scriptsize,text=ABink!70] at (5.35,-1.45)
    {(b) port schedule for $M_k$};
\end{tikzpicture}
\end{minipage}
\caption{Annular amplification.  (a) The blue cross-edges realize exactly
the six faces in \eqref{eq:annulusfaces}.  (b) For the ordinary joins, $P_i$
is paired with $R_{i+1}$; the last join pairs $P_{k-1}$ with $P_k$.  Each
copied path $W_i$ (green) misses both ports, so the annuli add no edge between
selected witnesses.}
\label{fig:annulus-family}
\end{figure}
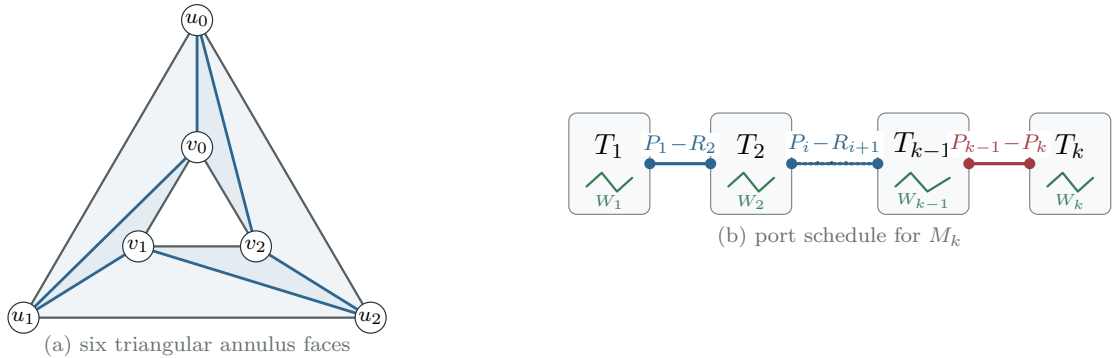

Every seed uses its port $P$ exactly once, while $R$ is used precisely in
seeds $2,\ldots,k-1$.  Since $P$ and $R$ are disjoint in a seed, no vertex
lies on two annuli.

\begin{lemma}[Topology]\label{lem:topology}
For every $k\ge2$, $M_k$ is a simple sphere triangulation with
\[
 |V(M_k)|=31k.
\]
\end{lemma}

\begin{proof}
 At each join, remove one open triangular face from the sphere assembled so
 far and one from the next unused seed sphere, then insert the triangulated
 cylinder
\eqref{eq:annulusfaces}.  The result is again a sphere.  Inductively, every
subsequent port is still facial because the two ports in an intermediate
seed are disjoint.  The copies are vertex-disjoint and each interface is new,
so no loop, parallel edge, or repeated face is introduced.

Each of the $k-1$ joins adds six edges, removes two faces, and adds six
faces.  Since $T$ has $31$ vertices, $87$ edges, and $58$ faces,
\[
 |E(M_k)|=87k+6(k-1)=93k-6
\]
and
\[
 |F(M_k)|=58k+4(k-1)=62k-4.
\]
The vertex count is unchanged by the joins.  The faces in
\eqref{eq:annulusfaces} also show directly that every old boundary edge and
every new cross-edge has two face incidences.
\end{proof}

\begin{lemma}[Exact induced-forest number]\label{lem:familyforest}
For every $k\ge2$,
\[
 a(M_k)=15k.
\]
\end{lemma}

\begin{proof}
Let $F$ induce a forest in $M_k$.  No annulus adds an edge within a seed, so
$M_k[V(T_i)]=T_i$.  Hence $F\cap V(T_i)$ induces a forest in $T_i$ and has at
most $15$ vertices by Proposition~\ref{prop:seed}.  Summing over the disjoint
seed vertex sets gives $|F|\le15k$.

For the reverse inequality, take the path witness $W$ in every seed.  Both
ports in \eqref{eq:ports} avoid $W$, so no annulus edge has a selected
endpoint on both sides.  The copied paths therefore induce a disjoint union
of forests on $15k$ vertices.
\end{proof}

\begin{lemma}[Minimum degree]\label{lem:mindegree}
For every $k\ge2$,
\[
 \delta(M_k)=5.
\]
\end{lemma}

\begin{proof}
The seed $T$ has a unique vertex of degree four, namely vertex~$4$, and this
vertex lies in the port $P=(0,4,6)$.  Every seed uses its port $P$ exactly once,
and every boundary vertex of an annulus gains two cross-neighbours.  Thus the
unique degree-four vertex in each seed has degree six in $M_k$.  All other seed
vertices have degree at least five, and the annular joins only add edges.
Moreover, the degree-five vertices of $T$ lie outside both ports and retain
degree five.  Hence $\delta(M_k)=5$.
\end{proof}

\begin{remark}[Degree data]\label{rem:degrees}
For completeness, the full degree counts are
\[
\begin{array}{c|rrrrr}
d&5&6&7&8&9\\ \hline
n_d&17k&5k+4&4k+2&2k-4&3k-2.
\end{array}
\]
These counts follow from the seed multiset $4^1 5^{17}6^6 7^7$: using $P$ in
every seed changes one degree $4$ to $6$ and two degrees $7$ to $9$, while
using $R$ in exactly $k-2$ seeds changes two degrees $6$ to $8$ and one degree
$7$ to $9$.  Thus $\Delta(M_k)=9$.  The counts also give the consistency check
$\sum_v\deg(v)=186k-12=2(93k-6)$.
\end{remark}

\begin{remark}[Separating triangle and connectivity]\label{rem:connectivity}
Lemma~\ref{lem:topology} implies that $M_k$ is maximal planar and therefore
$3$-connected.  The triangle $P_1$ is separating: removing its three vertices
separates the $28$ vertices of $T_1-P_1$ from all later seed vertices.  Hence
$\kappa(M_k)=3$.
\end{remark}

\begin{proof}[Proof of Theorem~\ref{thm:main}]
Lemma~\ref{lem:topology} gives a simple planar graph $M_k$ on $31k$ vertices,
Lemma~\ref{lem:familyforest} gives $a(M_k)=15k$, and
Lemma~\ref{lem:mindegree} gives $\delta(M_k)=5$.  The displayed ratio follows
immediately.
\end{proof}

\begin{corollary}
The Albertson--Berman conjecture is false.  In particular, it already fails
for simple planar graphs of minimum degree five.
\end{corollary}

\section{Discussion and open problems}

The counterexample rests on the structural role of the $14$-vertex
gadget~$X$.  Its terminal profile is the driving mechanism: the twelve
internal vertices contribute up to six to any induced forest when at most one
terminal is selected, but exactly five when both terminals are selected.
This one-unit drop is what makes the construction work.  Substituting $X$
onto a base edge converts an edge constraint in the base graph into a vertex
deficit in the substituted graph, and the transfer law
(Theorem~\ref{thm:transfer}) quantifies this conversion for an arbitrary
base graph and decorated edge set.

In this construction, leaving most base edges undecorated improves the ratio.
Decorating every edge (full-edge substitution) already pushes the ratio below
$\frac{1}{2}$, but for $B=K_4$ it gives $\frac{37}{76}\approx0.487$, above
$\frac{15}{31}$.  Undecorated base edges can impose cycle constraints without
incurring the $12$ internal vertices of an additional gadget.  For the
pentagonal bipyramid $C_5*\overline{K_2}$, only two of the fifteen edges are
decorated; the two apices restrict rim forests, the odd rim prevents selecting
all rim vertices, and the matching $\{01,23\}$ penalizes every rim subset of
four or more vertices.  Remark~\ref{rem:odd} shows that, within the odd-bipyramid
family $B_r=C_{2r+1}*\overline{K_2}$ with a maximum rim matching decorated, the
ratio $a_r/n_r$ is minimized at $r=2$.

The seed~$T$ has exactly one vertex of degree four, namely vertex~$4$.  The
annular amplification uses a port containing this vertex and raises the
minimum degree of every family member $M_k$ to five.  The full degree histogram
is given in Remark~\ref{rem:degrees}.

Borodin's acyclic $5$-coloring theorem gives $a(G)\ge2n/5$ for every planar
graph~$G$, while the present family gives the upper bound $15/31$ for the
optimal constant
\[
 c=\inf\frac{a(G)}{|V(G)|}.
\]
Thus
\[
 \frac25\le c\le\frac{15}{31}.
\]
Determining $c$, or improving either side of this interval, remains open.

We close with three specific open questions.

\begin{enumerate}
\item What is the exact value of~$c$?  The current range
      $2/5\le c\le 15/31$ is wide, and there is room to improve both the
      lower and upper bounds.

\item Does the Albertson--Berman inequality fail for $4$-connected planar
      graphs?  The graphs $M_k$ contain separating triangles; in particular,
      $P_1$ separates $T_1-P_1$ from the remaining seeds
      (Remark~\ref{rem:connectivity}).  Thus this construction does not
      address the $4$-connected case.

\item What is the smallest counterexample to the conjecture?  The seed~$T$
      has $31$ vertices, but this is the minimum within our construction;
      smaller counterexamples from other constructions may well exist.
\end{enumerate}

\section{Acknowledgments and AI disclosure}

The initial gadget and several proof ideas emerged from exploratory sessions 
with OpenAI's GPT-5.6 Sol, which also assisted in preparing the manuscript 
and verification code. The author is responsible for the problem selection, 
the mathematical verification, and the final form of all proofs, and assumes 
full responsibility for the correctness and content of the paper.

\appendix
\section{Computational verification}\label{app:verification}

The proof above is entirely symbolic.  The accompanying script
\path{verify_stronger_ab_family.py} independently reconstructs the
$31$-vertex seed from the data stated in the paper: the $14$-vertex gadget
rotation system, the pentagonal bipyramid, the decorated edges $01$ and
$23$, the stated vertex labelling, and the two completion edges $6\,12$
and $6\,20$.  It then computes the maximum induced-forest order of the
reconstructed seed directly by two independent exact algorithms: a
standard-library branch-and-bound minimum feedback vertex set solver and
a $0$--$1$ integer linear program with iteratively separated cycle cuts,
implemented with SciPy/HiGHS.  Both return
\[
 a(T)=15.
\]
Neither optimization uses the terminal profile, the transfer theorem, the
displayed $15$-vertex witness, or the sphere certificate to obtain its
optimum.

The script also retains the finite checks used in the earlier verifier:
the gadget embedding and its four terminal-conditioned optimizations over
$2^{12}$ internal subsets, the $2^7$ subsets of the
pentagonal-bipyramid core, the explicit seed witness and sphere
certificate, the two facial ports, and the seed degrees. For any supplied
$k\ge2$, the script constructs the corresponding annular member and checks 
its sphere certificate, the copied $15k$-vertex forest witness, the degree 
histogram, and an explicit separating three-set.  The script does not directly
optimize the full graph $M_k$; the upper bound $a(M_k)\le15k$, the
universal quantifier over $k$, and the separating-triangle observation
are supplied by the arguments in the body.  Thus the computation is redundant
independent certification of the finite seed and construction data rather
than a premise of the main theorem.

\end{document}